\documentclass[11pt,leqno]{amsart}%
\usepackage{amsmath}
\usepackage{amsfonts}
\usepackage{amssymb}
\usepackage{graphicx}%
\providecommand{\U}[1]{\protect\rule{.1in}{.1in}}
\newtheorem{theorem}{Theorem}

\newtheorem{corollary}[theorem]{Corollary}

\newtheorem{lemma}[theorem]{Lemma}

\newtheorem{remark}[theorem]{Remark}

\begin{document}

\title[Global comparison principles for the $p$-Laplace operator]{Global comparison principles for the $p$-Laplace operator on Riemannian manifolds}
\date{\today}

\author{Stefano Pigola}
\address{Dipartimento di Fisica e Matematica\\
Universit\`a dell'Insubria - Como\\
via Valleggio 11\\
I-22100 Como, ITALY}
\email{stefano.pigola@uninsubria.it}

\author{Giona Veronelli}
\address{Dipartimento di Matematica\\
Universit\`a di Milano\\
via Saldini 50\\
I-20133 Milano, ITALY}
\email{giona.veronelli@unimi.it}


\begin{abstract}
We prove global comparison results for the $p$-Laplacian on a $p$-parabolic manifold.
These involve both real-valued and vector-valued maps with finite $p$-energy.
\end{abstract}

\maketitle

\section{Introduction}

Let $\left(  M,\left\langle ,\right\rangle \right)  $ be a connected,
$m$-dimensional manifold and let $p>1$. Recall that the $p$-Laplacian of a
real valued function $u:M\rightarrow\mathbb{R}$ is defined by $\Delta
_{p}u=\operatorname{div}\left(  \left\vert \nabla u\right\vert ^{p-2}\nabla
u\right).$ The function $u$ is said to be $p$-subharmonic if $\Delta
_{p}u\geq0$. In case any bounded above, $p$-subharmonic function is
necessarily constant we say that the manifold $M$ is $p$-parabolic. It is
known that $p$-parabolicity of a complete manifold $M$ is related to volume
growth properties of the underlying manifold. Namely, $M$ is $p$-parabolic
provided, for some $x\in M$,%
\[
\left(  \frac{1}{\mathrm{vol}_{m-1}\partial B_{r}\left(  x\right)  }\right)
^{\frac{1}{p-1}}\notin L^{1}\left(  +\infty\right)  ,
\]
where $\partial B_{r}\left(  x\right)  $ denotes the metric sphere centered at
$x$, of radius $r>0$, and $\mathrm{vol}_{m-1}$ is the $\left(  m-1\right)
$-dimensional Hausdorff measure; \cite{RS}. Thus, for instance, the standard
Euclidean space $\mathbb{R}^m$ is $p$-parabolic if $m \leq p$.

Now, suppose that $M$ is $p$-parabolic, with $p\geq2$. It is known,
\cite{PRS-MathZ}, that a $p$-subharmonic function $u:M\rightarrow\mathbb{R}$
with finite $p$-energy $\left\vert \nabla u\right\vert \in L^{p}\left(
M\right)  $ must be constant. We shall show that this is nothing but a very
special case of a genuine comparison principle for the $p$-Laplace operator.

\begin{theorem}
\label{th_comparison}Let $\left(  M,\left\langle ,\right\rangle \right)  $ be
a connected, $p$-parabolic manifold, $p\geq2$. Assume that the smooth
functions $u,v:M\rightarrow\mathbb{R}$ satisfy%
\[
\Delta_{p}u\geq\Delta_{p}v\text{ on }M,
\]
and%
\[
\left\vert \nabla u\right\vert ,\left\vert \nabla v\right\vert \in
L^{p}\left(  M\right)  .
\]
Then, $u=v+A$ on $M$, for some constant $A\in\mathbb{R}$.
\end{theorem}

Besides real-valued functions one is naturally led to consider manifold-valued
maps. Several topological questions are related to the $p$-Laplacian of maps;
\cite{Wei}. Recall that the $p$-Laplacian (or the $p$-tension field) of a map
$u:M\rightarrow N$ between Riemannian manifolds is defined by%
\[
\Delta_{p}u=\operatorname{div}\left(  \left\vert du\right\vert ^{p-2}%
du\right)  ,
\]
Here, $du\in T^{\ast}M\otimes u^{-1}TN$ denotes the differential of $u$ and the
bundle $T^{\ast}M\otimes u^{-1}TN$ is endowed with its Hilbert-–Schmidt scalar
product $\left\langle ,\right\rangle$. Moreover, $-\operatorname{div}$ stands for the formal adjoint of the
exterior differential $d$ with respect to the standard $L^{2}$ inner product
on vector-valued $1$-forms. Say that $u$ is $p$-harmonic if $\Delta_{p}u=0$.
In \cite{SY}, Schoen and Yau prove a general comparison principle for
homotopic ($2$-)harmonic maps with finite ($2$-)energy into non-positively
curved targets. They assume that the complete, non compact manifold $M$ has
finite volume but the request that $M$ is ($2$-)parabolic suffices,
\cite{PRS-MathZ}. In this direction, comparisons for homotopic \thinspace
$p$-harmonic maps with finite $p$-energy into non-positively curved manifolds
are far from being completely understood. Some progress in the special
situation of a single map homotopic to a constant has been made in
\cite{PRS-MathZ}. In this note, we focus our attention on the case
$N=\mathbb{R}^{n}$. According to \cite{PRS-MathZ}, it is clear that, if $M$ is
$p$-parabolic, then every $p$-harmonic map $u:M\rightarrow\mathbb{R}^n$ with
finite $p$-energy $\left\vert du\right\vert \in L^{p}\left(  M\right)  $ must
be constant. However, using the very special structure of $\mathbb{R}^{n}$, we
are able to extend this conclusion, thus establishing a comparison principle
for maps $u,v:M\rightarrow\mathbb{R}^{n}$ having the same $p$-Laplacian. This
can be considered as a further step towards the comprehension of the general
comparison problem alluded to above.

\begin{theorem}
\label{th_rn} Suppose that $\left(  M,\left\langle ,\right\rangle \right)  $
is $p$-parabolic, $p\geq2$. Let $u,v:M\rightarrow\mathbb{R}^{n}$ be smooth
maps satisfying%
\begin{equation}
\Delta_{p}u=\Delta_{p}v\text{ on }M, \label{assumption_pDelta}%
\end{equation}
and%
\[
\left\vert du\right\vert ,\left\vert dv\right\vert \in L^{p}\left(  M\right)
.
\]
If $\left(  M,\left\langle ,\right\rangle \right)  $ is
$p$-parabolic then $\ u=v+A\mathrm{,}$ for some constant $A\in\mathbb{R}$.
\end{theorem}

\subsection*{Acknowledgement}
The authors express their gratitude to A.G. Setti for useful discussions during the preparation of this paper.

\section{Main tools}

In the proofs of Theorems \ref{th_comparison} and \ref{th_rn} we will make an
essential use of two main ingredients: (a) a version for the $p$-Laplacian of a
classical inequality for the mean-curvature
operator; (b) a parabolicity criterion which, in a sense, can be considered as
a global form of the divergence theorem in non-compact settings.

\subsection{A key inequality}

The following basic inequality was discovered by Lindqvist, \cite{L}.

\begin{lemma}
Let $\left(  V,\left\langle ,\right\rangle \right)  $ be a finite dimensional,
real vector space endowed with a positive definite scalar product. Let
$p\geq2$. Then, for every $x,y\in V$ it holds%
\[
\left\vert x\right\vert ^{p}+\left(  p-1\right)  \left\vert y\right\vert
^{p}\geq p\left\vert y\right\vert ^{p-2}\left\langle x,y\right\rangle
+\frac{1}{2^{p-1}-1}\left\vert x-y\right\vert ^{p}.
\]

\end{lemma}

As consequence, we deduce the validity of the next

\begin{corollary}
\label{coro_lind} In the above assumptions, for every $x,y\in V$, it holds%
\begin{equation}
\left\langle \left\vert x\right\vert ^{p-2}x-\left\vert y\right\vert
^{p-2}y,x-y\right\rangle \geq\frac{2}{p\left(  2^{p-1}-1\right)  }\left\vert
x-y\right\vert ^{p}. \label{MHCK}%
\end{equation}

\end{corollary}

\begin{proof}
We start computing
\[
\left\langle \left\vert x\right\vert ^{p-2}x -\left\vert y\right\vert ^{p-2}y    ,x-y\right\rangle
=\left\vert x\right\vert ^{p}+\left\vert y\right\vert ^{p}-\left\langle
x,y\right\rangle \left(  \left\vert x\right\vert ^{p-2}+\left\vert
y\right\vert ^{p-2}\right)  .
\]
On the other hand, applying twice Lindqvist inequality with the role of $x$
and $y$ interchanged we get%
\[
p\left(  \left\vert x\right\vert ^{p}+\left\vert y\right\vert ^{p}\right)
\geq p\left(  \left\vert x\right\vert ^{p-2}+\left\vert y\right\vert
^{p-2}\right)  \left\langle x,y\right\rangle +\frac{2}{\left(  2^{p-1}%
-1\right)  }\left\vert x-y\right\vert ^{p}.
\]
Inserting into the above completes the proof.
\end{proof}

\begin{remark}
\rm{Inequality (\ref{MHCK}) can be considered as a version for the $p$-Laplacian
of the classical Mikljukov-Hwang-Collin-Krust inequality; \cite{M}, \cite{Hw},
\cite{CK}. This latter states that, for every $x,y\in V$,%
\[
\left\langle \tfrac{x}{\sqrt{1+\left\vert x\right\vert ^{2}}}-\tfrac{y}%
{\sqrt{1+\left\vert y\right\vert ^{2}}},x-y\right\rangle \geq\tfrac
{\sqrt{1+\left\vert x\right\vert ^{2}}+\sqrt{1+\left\vert y\right\vert ^{2}}%
}{2}\left\vert \tfrac{x}{\sqrt{1+\left\vert x\right\vert ^{2}}}-\tfrac
{y}{\sqrt{1+\left\vert y\right\vert ^{2}}}\right\vert ^{2},
\]
equality holding if and only if $x=y$.}
\end{remark}

\subsection{The Kelvin-Nevanlinna-Royden criterion}

There is a very useful characterization of (non) $p$-parabolicity in terms of
special vector fields on the underlying manifold. It goes under the name of
Kelvin-Nevanlinna-Royden criterion. In the linear setting $p=2$ it was proved
in a paper by Lyons and Sullivan, \cite{LS}. See also \cite{PRS-Progress}. The
following non-linear extension is due to Gol'dshtein and Troyanov, \cite{GT}.

\begin{theorem}
\label{th_gt}The manifold $M$ is not $p$-parabolic if and only if there exists
a vector field $X$ on $M$ such that:\medskip

(a) $\left\vert X\right\vert \in L^{\frac{p}{p-1}}\left(  M\right)  $\medskip

(b) $\operatorname{div}X\in L_{loc}^{1}\left(  M\right)  $ and $\min\left(
\operatorname{div}X,0\right)  =\left(  \operatorname{div}X\right)  _{-}\in
L^{1}\left(  M\right)  $\medskip

(c) $0<\int_{M}\operatorname{div}X\leq+\infty$.
\end{theorem}

In particular, if $M$ is $p$-parabolic and $X$ is a vector field satisfying
(a) $\left\vert X\right\vert \in L^{\frac{p}{p-1}}\left(  M\right)  $, (b)
$\operatorname{div}X\in L_{loc}^{1}\left(  M\right)  $, and (c)
$\operatorname{div}X\geq0$ on $M$, then we must necessarily conclude that
$\operatorname{div}X=0$ on $M$.

\section{Proofs of the comparison principles}

We are now in the position to prove the main results.

\begin{proof}
[Proof (of Theorem \ref{th_comparison})]Fix any \thinspace$x_{0}\in M$, let
$A=u\left(  x_{0}\right)  -v\left(  x_{0}\right)  $ and define $\Omega_{A}$ to
be the connected component of the open set%
\[
\left\{  x\in M:A-1<u\left(  x_{0}\right)  -v\left(  x_{0}\right)
<A+1\right\}
\]
which contains $x_{0}$. By standard topological arguments, $\Omega_{A}%
\neq\emptyset$ is a (connected) open set. Let $\alpha:\mathbb{R}%
\rightarrow\mathbb{R}_{\geq0}$ be the piece-wise linear function defined by%
\[
\alpha\left(  t\right)  =\left\{
\begin{array}
[c]{ll}%
0 & t\leq A-1\\
\left(  t-A+1\right)  /2 & A-1\leq t\leq A+1\\
1 & t\geq A+1.
\end{array}
\right.
\]
Consider the vector field%
\[
X=\alpha\left(  u-v\right)  \left\{  \left\vert \nabla u\right\vert
^{p-2}\nabla u-\left\vert \nabla v\right\vert ^{p-2}\nabla v\right\}  .
\]
A direct computation gives%
\begin{align*}
\operatorname{div}X  &  =\alpha^{\prime}\left(  u-v\right)  \left\langle
\left\vert \nabla u\right\vert ^{p-2}\nabla u-\left\vert \nabla v\right\vert
^{p-2}\nabla v,\nabla u-\nabla v\right\rangle \\
&  +\alpha\left(  u-v\right)  \left(  \Delta_{p}u-\Delta_{p}v\right) \\
&  \geq\frac{2}{2^{p-1}-1}\alpha^{\prime}\left(  u-v\right)  \left\vert \nabla
u-\nabla v\right\vert ^{p},
\end{align*}
where in the last inequality we have used Corollary \ref{coro_lind}, the fact
that $\alpha,\alpha^{\prime}\geq0$ and the assumption $\Delta_{p}u-\Delta
_{p}v\geq0$. Therefore%
\[
\operatorname{div}X\geq0,\text{ on }M,
\]
the equality holding if and only if%
\[
\alpha^{\prime}\left(  u-v\right)  \left\vert \nabla u-\nabla v\right\vert
=0.
\]
On the other hand, for a suitable constant $C>0$,%
\[
\left\vert X\right\vert ^{\frac{p}{p-1}}\leq C\left(  \left\vert \nabla
u\right\vert ^{p}+\left\vert \nabla v\right\vert ^{p}\right)  \in L^{1}\left(
M\right)  .
\]
Therefore, Theorem \ref{th_gt} yields%
\[
\operatorname{div}X=0\text{ on }M.
\]
Since $\alpha^{\prime}\left(  u-v\right)  \neq0$ on $\Omega_{A}$, we deduce%
\[
u-v\equiv A\text{, on }\Omega_{A}.
\]
It follows that the open set $\Omega_{A}$ is also closed. Since $M$ is
connected we must conclude that $\Omega_{A}=M$ and $u-v=A$ on $M$.
\end{proof}

\begin{proof}
[Proof (of Theorem \ref{th_rn})]We suppose that either $u$ or $v$ is non-constant, for otherwise there's nothing to prove. Fix $q_{0}\in M$. Set $C:=u(q_{0})-v(q_{0}%
)\in\mathbb{R}^{n}$ and introduce the radial function $r:\mathbb{R}%
^{n}\rightarrow\mathbb{R}$ defined as $r(x)=|x-C|$. For $T>0$, consider the
piecewise differentiable vector field $X_{T}$ defined as
\[
X_{T}:=\left[  \left.  dh_{T}\right\vert _{(u-v)}\circ\left(  |du|^{p-2}%
du-|dv|^{p-2}dv\right)  \right]  ^{\sharp}%
\]
where $h_{T}\in C^{1}(\mathbb{R}^{n},\mathbb{R})$ is the function
\[
h_{T}(x):=\left\{
\begin{array}
[c]{ll}%
\frac{r^{2}(x)}{2} & \text{if }r(x)<T\\
Tr(x)-\frac{T^{2}}{2} & \text{if }r(x)\geq T.
\end{array}
\right.
\]
We observe that $h_{T}\in C^{2}$ where $r(x)\neq T$ and that $X_{T}$ is well
defined since there exists a canonical identification
\[
T_{(u-v)(q)}\mathbb{R}^{n}\cong T_{u(q)}\mathbb{R}^{n}\cong T_{v(q)}%
\mathbb{R}^{n}\cong\mathbb{R}^{n}.
\]
 Wa also observe that, by Sard theorem, for a.e.
$T>0$, the level set $\left\{  |u-v-C|=T\right\}  $ is a smooth (possibly
empty) hypersurface, hence a set of measure zero. Thus, the vector field $X_T$ is
weakly differentiable and, for a.e. $T>0$, the weak divergence of $X_T$ is given by 
\begin{align*}
\operatorname{div}X_{T} &=d(\tfrac{r^2}{2})|_{(u-v)}\circ(\Delta_{p}u-\Delta
_{p}v) \\
&+ {}^{M}\operatorname{tr}\left(  \operatorname{Hess}(\tfrac{r^2}{2})|_{(u-v)}\left(
du-dv,|du|^{p-2}du-|dv|^{p-2}dv\right)  \right),
\end{align*}
if  $r(x)<T$, and
\begin{align*}
\operatorname{div}X_{T} &= d(Tr)|_{(u-v)}\circ(\Delta_{p}u-\Delta_{p}v) \nonumber\\
&+ {}^{M}\operatorname{tr}\left(  \operatorname{Hess}(Tr)|_{(u-v)}\left(du-dv,|du|^{p-2}du-|dv|^{p-2}dv\right)  \right).
\end{align*}
if  $r(x) \geq T$. The first term on the RHS vanishes by assumption. Moreover, by standard
computations, we have $\operatorname{Hess}(r)=r^{-1}(\left\langle
,\right\rangle _{\mathbb{R}^{n}}-dr\otimes dr)$ on $\mathbb{R}^{n}%
\setminus\left\{  C\right\}  $. Thus,
\[
\begin{array}
[c]{ll}%
\operatorname{Hess}(\frac{r^2}{2})=dr\otimes dr+r\operatorname{Hess}(r)=\left\langle ,\right\rangle
_{\mathbb{R}^{n}} & \text{if }r(x)<T,\\
\operatorname{Hess}(Tr)=T\operatorname{Hess}(r)=\frac{T}{r}(\left\langle ,\right\rangle _{\mathbb{R}%
^{n}}-dr\otimes dr\rangle) & \text{if }r(x)\geq T.
\end{array}
\]
As a consequence, for $q\in M$ such that $r((u-v)(q))<T$, by Corollary
\ref{coro_lind} we get
\begin{equation}
\operatorname{div}X_{T}=\left\langle du-dv,|du|^{p-2}du-|dv|^{p-2}%
dv\right\rangle \geq\frac{|du-dv|^{p}}{p(2^{p-1}-1)},\label{dive_r2}%
\end{equation}
while, for $q\in M$ such that $r((u-v)(q))\geq T$, it holds
\begin{align}
\operatorname{div}X_{T} &  =\frac{T}{r(u-v)}\left\langle du-dv,|du|^{p-2}%
du-|dv|^{p-2}dv\right\rangle \label{dive_r}\\
&  -\frac{T}{r(u-v)}\left\langle dr|_{(u-v)}(du-dv),dr|_{(u-v)}(|du|^{p-2}%
du-|dv|^{p-2}dv)\right\rangle \nonumber\\
&  \geq\frac{T}{r(u-v)}\frac{|du-dv|^{p}}{p(2^{p-1}-1)}-(|du|+|dv|)(|du|^{p-1}%
+|dv|^{p-1})\nonumber\\
&  \geq\frac{T}{r(u-v)}\frac{|du-dv|^{p}}{p(2^{p-1}-1)}-(|du|^{p}%
+|dv|^{p}+|du|^{p-1}|dv|+|dv|^{p-1}|du|)\nonumber\\
&  \geq\frac{T}{r(u-v)}\frac{|du-dv|^{p}}{p(2^{p-1}-1)}-2(|du|^{p}%
+|dv|^{p}),\nonumber
\end{align}
where we have used again Corollary \ref{coro_lind} for the first term and
Cauchy-Schwarz inequality, Young's inequality and the facts that $|dr|=1$ and
$r(u-v)\geq T$ for the second one. Let us now compute the $L^{\frac{p}{p-1}}%
$-norm of $X_{T}.$ Since
\[
\left\vert |du|^{p-2}du-|dv|^{p-2}dv\right\vert ^{\frac{p}{p-1}}\leq\left(
|du|^{p-1}+|dv|^{p-1}\right)  ^{\frac{p}{p-1}}\leq2^{\frac{1}{p-1}}\left(
|du|^{p}+|dv|^{p}\right)  ,
\]
we have
\begin{align*}
\int_{\left\{  |u-v-C|<T\right\}  }|X_{T}|^{\frac{p}{p-1}} &  \leq
\int_{\left\{  |u-v-C|)<T\right\}  }|u-v-C|^{\frac{p}{p-1}}\left\vert
|du|^{p-2}du-|dv|^{p-2}dv\right\vert ^{\frac{p}{p-1}}\\
&  \leq T^{\frac{p}{p-1}}2^{\frac{1}{p-1}}\left(  \left\Vert du\right\Vert
_{p}^{p}+\left\Vert dv\right\Vert _{p}^{p}\right)  <+\infty
\end{align*}
and
\begin{align*}
\int_{\left\{  |u-v-C|)>T\right\}  }|X_{T}|^{\frac{p}{p-1}} &  \leq
\int_{\left\{  |u-v-C|>T\right\}  }T^{\frac{p}{p-1}}\left\vert |du|^{p-2}%
du-|dv|^{p-2}dv\right\vert ^{\frac{p}{p-1}}\\
&  \leq T^{\frac{p}{p-1}}2^{\frac{1}{p-1}}\left(  \left\Vert du\right\Vert
_{p}^{p}+\left\Vert dv\right\Vert _{p}^{p}\right)  <+\infty.
\end{align*}
Hence $X_{T}$ is a weakly differentiable vector field with $|X_{T}|\in
L^{\frac{p}{p-1}}(M)$ and $\operatorname{div}X_{T}\in L_{loc}^{1}(M)$. To
apply Theorem \ref{th_gt}, it remains to show that $(\operatorname{div}%
X_{T})_{-}\in L^{1}(M)$. By inequalities (\ref{dive_r2}) and (\ref{dive_r}),
we deduce that
\begin{equation}
\int_{M}\left\vert \left(  \operatorname{div}X_{T}\right)  _{-}\right\vert
\leq2\int_{\left\{  |u-v-C|>T\right\}  }(|du|^{p}+|dv|^{p})\leq2(\left\Vert
du\right\Vert _{p}^{p}+\left\Vert dv\right\Vert _{p}^{p})<+\infty
.\label{div_-}%
\end{equation}
Then, the assumptions of Theorem \ref{th_gt} are satisfied and we get, for
\ a.e. $T>0$,
\[
\int_{M}\operatorname{div}X_{T}\leq0.
\]
According to (\ref{div_-}) we now choose a sequence $T_{n}\nearrow+\infty$
such that
\[
\int_{M}\left\vert \left(  \operatorname{div}X_{T_{n}}\right)  _{-}\right\vert
\leq2\int_{\left\{  |u-v-C|>T_{n}\right\}  }(|du|^{p}+|dv|^{p})<\frac{1}{n}.
\]
As a consequence,%
\begin{align}
\int_{\left\{  |u-v-C|<T_{n}\right\}  }\frac{|du-dv|^{p}}{p(2^{p-1}-1)} &
\leq\int_{\left\{  |u-v-C|<T_{n}\right\}  }\left(  \operatorname{div}X_{T_{n}%
}\right)  _{+}\label{div_eps}\\
&  \leq\int_{M}\left(  \operatorname{div}X_{T_{n}}\right)  _{+}\nonumber\\
&  \leq-\int_{M}\left(  \operatorname{div}X_{T_{n}}\right)  _{-}<\frac{1}%
{n}.\nonumber
\end{align}
Therefore, letting  $n$ go to $+\infty$,
we obtain
\[
\int_{M}\frac{|d(u-v)|^{p}}{p(2^{p-1}-1)}=0,
\]
that is, $u-v\equiv u(q_{0})-v(q_{0})=C$ on $M$.
\end{proof}


\bigskip

\begin{thebibliography}{99}                                                                                               %


\bibitem {CK}P. Collin, R. Krust. Le probl\`{e}me de Dirichlet pour l'equation
des surfaces minimales sur des domaines non born\'{e}s. Bull. Soc. Math.
France, \textbf{119} (1991), 443-458.

\bibitem {GT}V. Gol'dshtein, M. Troyanov,\textit{\ The
Kelvin-Nevanlinna-Royden criterion for }$p$\textit{-parabolicity.} Math Z.
\textbf{232} (1999), 607--619.

\bibitem {Hw}J.-F. Hwang, \textit{Comparison principles and Liouville theorems
for prescribed mean curvature equations in unbounded domains.} Ann. Scuola
Norm. Sup. Pisa Cl. Sci. \textbf{15} (1988), 341--355.

\bibitem {L}P. Lindqvist, \textit{On the equation }$\operatorname{div}\left(
\left\vert \nabla u\right\vert ^{p-2}\nabla u\right)  +\lambda\left\vert
u\right\vert ^{p-2}u=0$. Proc. Amer Math. Soc. \textbf{109} (1996), 157--164.

\bibitem {LS}T. Lyons, D. Sullivan,\textit{\ Function theory, random paths and
covering spaces.} J. Diff. Geom. \textbf{18} (1984), 229--323.

\bibitem {M}V.M. Mikljukov. A new approach to Bernstein theorem and to related
questions for equations of minimal surface type (in Russian). Mat. Sb.
\textbf{108} (1979), 268-289; Engl. transl. in Math. USSR Sb., \textbf{36}
(1980), 251-271.


\bibitem {PRS-MathZ}S. Pigola, M. Rigoli, A.G. Setti, \textit{Constancy of
p-harmonic maps of finite q-energy into non-positively curved manifolds.}
Math. Z. \textbf{258} (2008), 347--362.

\bibitem {PRS-Progress}S. Pigola, M. Rigoli, A.G. Setti, \textit{Vanishing and
finiteness results in geometric analysis: a generalization of the Bochner
technique.} Progress in Mathematics \textbf{266} (2008), Birkh\"{a}user.

\bibitem {RS}M. Rigoli, A.G. Setti,\textit{\ Liouville type theorems for
}$\varphi$\textit{-subharmonic functions.} Rev. Mat. Iberoamer. \textbf{17}
(2001), 471--450.

\bibitem {SY}R. Schoen, S.T. Yau, \textit{Compact group actions and the
topology of manifolds with nonpositive curvature.} Topology \textbf{18}
(1979), 361--380.

\bibitem {Wei}S.W. Wei,\textit{\ Representing homotopy groups and spaces of
maps by p-harmonic maps.} Indiana Math. J. \textbf{47} (1998), 625--669.
\end{thebibliography}
\end{document}